\documentclass[a4paper,11pt,oneside]{article}

\usepackage{graphicx}
\usepackage{amssymb}
\usepackage{amsthm}
\usepackage{amsmath}
\usepackage{enumitem}
\usepackage{graphicx}
\usepackage{psfrag}

\usepackage[
pdfauthor={Markus Haltmeier},
colorlinks
]{hyperref}
\usepackage{breakurl}

\newtheoremstyle{TheoStyle}
    {1em}                      
    {1em}                      
    {\itshape}                      
    {}                      
    {\bfseries}     
    {.\\}                      
    {1em}              
    {}                      

\newtheoremstyle{TheoStyle2}
    {1em}                      
    {1em}                      
    {}                      
    {}                      
    {\bfseries}     
    {.\\}                      
    {1em}                     
    {}                      

\newtheoremstyle{TheoStyle2i}
    {}                      
    {}                      
    {}                      
    {}                      
    {\bfseries}     
    {.\\}                      
    { }                     
    {}                      

\theoremstyle{TheoStyle}
\newtheorem{theorem}{Theorem}[section]

\newtheorem{lemma}[theorem]{Lemma}

\theoremstyle{TheoStyle2}\newtheorem{definition}[theorem]{Definition}

\usepackage{titlesec}

\newcommand{\coloneqq}{:=}
\newcommand{\R}{\mathbb R}

\newcommand{\edot}{\,\cdot\,}

\newcommand{\Mo}{\mathcal{M}}
\newcommand{\Bo}{\mathcal{B}}

\newcommand{\Ho}{\mathcal{H}}

\newcommand{\D}{\mathcal{D}}

\newcommand{\rmd}{\mathrm d}

\newcommand\abs[1]{\left\vert#1\right\vert}
\newcommand\sabs[1]{\lvert#1\rvert}

\newcommand\set[1]{\left\{#1\right\}}

\newcommand\sset[1]{\{#1\}}
\newcommand{\om}{\omega}

\DeclareMathOperator{\diag}{diag}

\newcommand{\kl}[1]{\left(#1\right)}
\newcommand{\ekl}[1]{\left[#1\right]}

\newcommand{\skl}[1]{(#1)}
\newcommand\inner[2]{\left\langle#1,#2\right\rangle}
\newcommand\sinner[2]{\langle#1,#2\rangle}
\newcommand{\req}[1]{(\ref{eq:#1})}

\usepackage{geometry}
\usepackage{bbding}

\allowdisplaybreaks

\numberwithin{equation}{section}
\numberwithin{figure}{section}
\numberwithin{theorem}{section}
\title{Exact Reconstruction Formula for the Spherical Mean Radon Transform on Ellipsoids}

\author{Markus Haltmeier}

\date{Department of Mathematics\\
University of Innsbruck\\
Technikestra{\ss}e 21a, A-6020 Innsbruck\\
E-mail:
\href{mailto:markus.haltmeier@uibk.ac.at}
{\tt markus.haltmeier@uibk.ac.at}}

\begin{document}

\maketitle

\begin{abstract}
Many modern imaging and remote sensing applications require reconstructing a function from spherical averages (mean values). Examples include   photoacoustic tomography, ultrasound imaging or SONAR. Several  formulas of the back-projection type for recovering a function in $n$ spatial dimensions from mean values over spheres centered on a sphere have been derived in [D. Finch, S. K. Patch, and Rakesh, SIAM J. Math. Anal. 35(5), pp. 1213--1240, 2004] for odd spatial dimension and in [D. Finch, M. Haltmeier, and Rakesh, SIAM J. Appl. Math. 68(2), pp. 392--412, 2007] for even spatial dimension. In this paper we generalize some of these formulas to the case where the centers of integration lie on the boundary of an arbitrary ellipsoid. For the special cases $n=2$ and $n=3$ our results have recently been established in [Y. Salman, J. Math. Anal. Appl., 2014, in press]. For the higher dimensional case $n > 3$ we  establish  proof techniques extending the ones in  the above references.

Back-projection type inversion formulas for recovering a function from  spherical means with centers on an ellipsoid  have  first been derived in  [F. Natterer, Inverse Probl. Imaging 6(2), pp. 315--320, 2012]  for $n=3$ and in [V. Palamodov,  Inverse Probl. 28(6), 065014, 2012] for arbitrary dimension. The results of Natterer have later been generalized to arbitrary  dimension in  [M. Haltmeier,  SIAM J. Math. Anal. 46(1), pp. 214--232, 2014].   Note that these formulas are different from the ones derived  in the present paper.


\bigskip\noindent{\bf AMS classification numbers.}
44A12, 45Q05, 35L05, 92C55.
\end{abstract}

\section{Introduction}
\label{sec:intro}

The spherical mean Radon transform $\Mo \colon C^\infty \kl{\R^n}\to C^\infty \kl{\R^n \times \kl{0, \infty}}$ maps a  smooth function $f\colon \R^n \to \R$ to the spherical mean values
\begin{equation} \label{eq:means}
\kl{\Mo f} \kl{z,r} = \frac{1}{\omega_{n-1}} \int_{S^{n-1}} f\kl{z + r\om} \rmd \omega\,,
\qquad
\text{ for } \kl{z,r} \in \R^n \times \kl{0,\infty}\,.
\end{equation}
Here $S^{n-1} = \sset{\omega \in \R^n \colon \sabs{\om}=1}$ is the $n-1$ dimensional unit sphere,  $\omega_{n-1} \coloneqq \sabs{S^{n-1}}$ its total surface measure, and $\rmd \omega$ the standard surface measure on $S^{n-1}$. The value $\kl{\Mo f} \kl{z,r}$ is the average of $f$ over a sphere with center $z\in \R^n$ and radius  $r >0$.

In this paper we  study the problem of recovering a function $f$ supported in an ellipsoid $\bar E \subset \R^n$ from the  spherical mean values $\kl{\Mo f} \kl{z,r}$ with centers restricted to the boundary of the ellipsoid. Recovering a function from spherical mean values with restricted centers is essential for many imaging and remote sensing applications, such as photoacoustic and thermoacoustic tomography (see \cite{BurBauGruHalPal07,FinRak09,KucKun08,XuWan06}), SONAR (see \cite{BelFel09,QuiRieSch11}), or ultrasound tomography (see \cite{Nor80,NorLin81}).

In \cite{FinHalRak07,FinPatRak04} several explicit reconstruction formulas for the spherical mean Radon transform have been derived for recovering a function supported in an $n$-dimensional ball from its spherical mean values centered on the boundary sphere. The proofs of  \cite{FinHalRak07,FinPatRak04} are based on integral identities established directly for $n=2$ and $n=3$. The higher dimensional cases have been reduced to the two and three dimensional cases  by expanding the function to be recovered  in a series of spherical harmonics.
In \cite{Sal14} one of the formulas of \cite{FinHalRak07,FinPatRak04} as well as the  methods of proofs  have been extended to elliptical domains in $\R^2$ and $\R^3$. In the present  paper we generalize these formulas to elliptical domains in arbitrary dimension. Opposed to \cite{FinHalRak07,FinPatRak04}  the  proofs for the higher dimensional cases are not based on a spherical harmonic expansion, which may be difficult to generalize to elliptical domains. Instead we present direct proofs that generalize  the proofs of \cite{FinHalRak07,FinPatRak04} from spherical center sets in two and three dimensions to elliptical center sets in arbitrary dimensions.

\subsection{Notation}
\label{sec:notation}

Throughout this paper, let $a_1, \dots, a_n  > 0 $ be given numbers,  let $A \coloneqq \diag (a_1, \dots, a_n)$
denote the diagonal matrix with diagonal entries
$a_i$, and let
\begin{equation}\label{eq:ell}
E
\coloneqq
\set{x \in \R^n: \abs{A^{-1}x} < 1}
= \set{x \in \R^n : \sum_{i=1}^n \frac{x_i^2}{a_i^2} < 1}
\end{equation}
denote the corresponding solid ellipsoid with semi-principal axes $a_1, \dots, a_n$.
In \req{ell} and below $\abs{\,\cdot\,}$ is the Euclidian norm on $\R^n$; the corresponding inner product will be denoted by $\inner{x}{y} = \sum_{i=1}^n  x_i  y_i$. Any point on the boundary of $E$ will be written in the form $A\sigma \in \partial E$ with
$\sigma \in S^{n-1}$.

We denote by $C^\infty \kl{\R^n}$  the set of all smooth (that is infinitely times differentiable) functions $f \colon \R^n \to \R$ and by
 $C^\infty_0 \kl{\bar E} \subset C^\infty \kl{\R^n}$ the subset of all  smooth functions  with support contained in  the
 closure $\bar E$.  The  spherical means $(\Mo f)\kl{z, r}$ of a function $f \in C^\infty \kl{\R^n}$ are defined by \req{means}. We are in particular interested in the case where $f \in C^\infty_0\kl{\bar E}$ and where
the centers of integration are restricted to $\partial E$.
Further we denote by
\begin{equation*}
	\Delta_{Ax} \coloneqq \sum_{i=1}^{n} \frac{1}{a_i^2}
	\frac{\partial^2}{\partial x_i^2}
\end{equation*}
the Laplacian with respect to the variable $Ax$ and by
$\D_r \coloneqq \frac{1}{2r} \frac{\partial}{\partial r}$
the differentiation operator with respect to $r^2$.
Occasionally we make use of the  Hilbert transform
$\Ho_s $  defined as the convolution with the distribution $\mathrm{P.V.} \kl{1/ \pi s}$.
We will also frequently write $r^{k}$ for the multiplication operator that maps a function $\kl{z, r} \mapsto g \kl{z, r}$ to
the function $\kl{z, r} \mapsto r^{k} g \kl{z, r}$.

An important role in our  analysis plays the function
$G_n \colon \R^{2n} \setminus \{ (x,x) : x \in \R^n\} \to \R$ defined by
\begin{equation*}
G_n \kl{x, y} =
\begin{cases}
\frac{1}{2\pi} \log \abs{x-y}
&  \text{ for } n=2
\\
\frac{\abs{x-y}^{2-n}}{\omega_{n-1} (2-n)}
& \text{ for } n > 2 \,,
\end{cases}
\end{equation*}
which  is the fundamental solution  of the $n$-dimensional
Laplacian $\Delta = \sum_{i=1}^n \frac{\partial^2}{\partial x_i^2}$.
By definition, the  fundamental solution $G_n $  is a solution of   $\Delta_x  G_n \kl{x,y} = \delta_n \kl{x-y}$,
with $\delta_n$ denoting the $n$-dimensional delta distribution.

\subsection{Main results}
\label{sec:means}

As the  main result of  this paper  we derive explicit reconstruction formulas for recovering a smooth function supported in the ellipsoid $\bar E$ from its spherical mean values with centers on the boundary $\partial E$.

In even dimension our main result reads as follows:

\begin{theorem}[Reconstruction in even dimension]  \label{thm:inv-even}
Let $n \geq 2$ be even, let $E$ be the solid ellipsoid defined by \req{ell}, and define the  constant $c_n \coloneqq (-1)^{(n-2)/2}  \omega_{n-2} \pi (n-2)! \, 2^{2-n} $.

Then, for every $f \in C^\infty_0\kl{\bar E}$ and every $x \in E$,
we have
\begin{multline}\label{eq:inv-even}
f\kl{x}
= \frac{\det\skl{A}}{c_n} \,
\Delta_{Ax}
\int_{S^{n-1}}
\int_{0}^\infty
\kl{ r  \D_r^{n-2} r^{n-2}  \Mo f }\kl{A\sigma, r}
\\ \times \log \abs{\sabs{x - A \sigma}^2 - r^2} \rmd r \rmd \sigma \,.
\end{multline}
Here $\partial_r$ denotes differentiation with respect to $r$,
$\D_r = \skl{2r}^{-1} \partial_r$ differentiation with respect to $r^2$, and $\Delta_{Ax}= \sum_{i=1}^{n} \frac{1}{a_i^2}
\partial^2_{x_i}$ the Laplacian with respect to $Ax$.
\end{theorem}

\begin{proof}
See Section \ref{sec:even}.
\end{proof}

In odd  dimensions we have the following
corresponding result:

\begin{theorem}[Reconstruction in odd dimension]  \label{thm:inv-odd}
Let $n \geq 3$ be odd, let $E$ be the  solid ellipsoid  defined by \req{ell}, and define the  constant
$c_n \coloneqq (-1)^{(n-1)/2}  \om_{n-2}(n-2)! \,  2^{3-n}
$.

Then, for every $f \in C^\infty_0\kl{\bar E}$ and every  $x \in E$,
we have
\begin{equation}\label{eq:inv-odd}
f\kl{x}
	=
	\frac{\det\skl{A}}{c_n}
	\Delta_{Ax} \,
	\int_{S^{n-1}}
	\kl{ r  \D_r^{n-3} r^{n-2} \Mo f }
	\kl{A\sigma, \sabs{x-A\sigma}}
	\rmd \sigma \,.
\end{equation}
Here $\partial_r$, $\D_r$ and $\Delta_{Ax}$ are as in
Theorem \ref{thm:inv-even}.
\end{theorem}

\begin{proof}
See Section \ref{sec:means-odd}.
\end{proof}

If all semi-principal axis $a_1, \dots, a_n$ coincide, then the ellipsoid $E$ is obviously an $n$-dimensional
ball. In such a case, the reconstruction formulas of
Theorem \ref{thm:inv-even} and \ref{thm:inv-odd} have  been first established in \cite{FinHalRak07} for even dimensions and in \cite{FinPatRak04} for odd dimensions. For the special cases of elliptical domains in two and three dimension, the formulas of Theorem \ref{thm:inv-even} and \ref{thm:inv-odd} have been recently established in~\cite{Sal14}. Different reconstruction formulas of the  back-projection type for spherical means on ellipsoids have recently been obtained for two and three spatial dimensions in \cite{AnsFilMadSey13,Hal13a,Nat12} and for arbitrary dimension in \cite{Hal14,Pal12}.

\subsection{Outline}

The remainder of this paper is mainly devoted to the proofs of
Theorems~\ref{thm:inv-even} and  \ref{thm:inv-odd}. In the following
Section \ref{sec:aux} we derive some required auxiliary lemmas.
The proof of Theorem~\ref{thm:inv-even} will be given in Section
\ref{sec:even} and the proof of Theorem \ref{thm:inv-odd} will be given
in Section \ref{sec:means-odd}. The paper concludes with a short discussion in Section
\ref{sec:discussion}.

\section{Auxiliary results}
\label{sec:aux}

Throughout the following, $A$ denotes the  diagonal matrix with positive diagonal entries  $a_1, \dots, a_n >0$, and $E$ denotes the  corresponding solid ellipsoid defined by~\req{ell}.

\begin{definition}[Reconstruction integral]
For $\Phi \in L^1_{\rm loc} \kl{\R}$ and
$f  \in  C^\infty_0 \kl{\bar E}$ we define
\begin{multline}  \label{eq:bp}
\kl{\Bo_{\Phi} \Mo f}  \kl{x} \coloneqq
\int_{S^{n-1}}
\int_{0}^\infty
r  \D_r^{n-2} r^{n-2} \kl{\Mo f }\kl{A\sigma, r}
\\ \times \Phi \kl{\sabs{x - A \sigma}^2 - r^2} \rmd r \rmd \sigma \,.
\end{multline}
\end{definition}

In the even dimensional case the reconstruction integral will
be applied with $\Phi\kl{s} = \log \abs{s}$ whereas in odd dimensions we use \req{bp} with $\Phi\kl{s} =
\chi \set{s > 0}$. In both cases we will show that
$\Delta_{Ax} \Bo_{\Phi} \Mo f $
is a constant multiple of $f$, which yields the reconstruction formulas
of Theorems \ref{thm:inv-even} and \ref{thm:inv-odd} (see Sections \ref{sec:even}
and \ref{sec:means-odd}).

\begin{lemma}[Basic integral identity]\label{lem:means}
Let $\Phi \in L^1_{\rm loc} \kl{\R}$ be locally integrable
and let $\Phi^{(n-2)}$ denote its $n-2$ fold distributional derivative.
Then, for every $f\in C^\infty_0 \kl{\bar E}$ and every $x \in \R^n$,
\begin{multline} \label{eq:int}
\kl{\Bo_{\Phi} \Mo f}  \kl{x}
= \frac{\omega_{n-2}}{\omega_{n-1}}\int_{\R^n}
f \kl{y}
\int_{-1}^{1} \kl{1-s^2}^{(n-3)/2}
\\ \times \Phi^{(n-2)} \kl{ 2 \abs{Ax-Ay}
\kl{\frac{\abs{x}^2-\abs{y}^2}{2\abs{A\kl{x-y}}} -   s}} \rmd s
\rmd y
 \,.
\end{multline}
\end{lemma}

\begin{proof}
Since $\D_r = \skl{2r}^{-1} \partial_r$, we have
$r  \D_r^{n-2} =  (-1)^{n-2} \skl{\D_r^{n-2}}^* r$ with $\skl{\D_r^{n-2}}^*$ denoting the formal $L^2$-adjoint of $\D_r^{n-2}$.
Integration by parts and one application of Fubini's theorem therefore yields
\begin{align*}
(\Bo_{\Phi} & \Mo f)\skl{x}
=
\int_{S^{n-1}}
\int_{0}^\infty
\kl{r  \D_r^{n-2} r^{n-2} \Mo f }\kl{A\sigma, r}
\Phi \kl{\sabs{x - A \sigma}^2- r^2} \rmd r \rmd \sigma
\\&=
 (-1)^{n-2}
 \int_{S^{n-1}}
\int_{0}^\infty
\kl{\D_r^{n-2}}^* \kl{ r^{n-1} \Mo f }\kl{A\sigma, r}
\Phi \kl{\sabs{x - A \sigma}^2- r^2} \rmd r \rmd \sigma
\\
&=
 (-1)^{n-2}
 \int_{S^{n-1}}
\int_{0}^\infty
r^{n-1} \kl{\Mo f }\kl{A\sigma, r}
\D_r^{n-2}
\Phi \kl{\sabs{x - A \sigma}^2- r^2} \rmd r \rmd \sigma
\\
&
=\int_{0}^\infty
r^{n-1} \kl{\Mo f }\kl{A\sigma, r}
\int_{S^{n-1}}
\Phi^{(n-2)} \kl{\sabs{x - A \sigma}^2- r^2} \rmd \sigma
\rmd r \,.
\end{align*}
The last identity, the definition of $\skl{\Mo f} \kl{A\sigma, r}$, and the use of spherical coordinates  $\kl{r, \omega} \mapsto
A \sigma + r \omega$ around the center $A \sigma$ show
\begin{equation}\label{eq:int1}
\kl{\Bo_{\Phi} \Mo f}  \kl{x}= \frac{1}{\omega_{n-1}}
\int_{\R^n} f \kl{y}
\int_{S^{n-1}}
\Phi^{(n-2)} \kl{\sabs{x - A \sigma}^2 -
\sabs{y - A \sigma}^2} \rmd \sigma
\rmd y
\end{equation}

Next we write the argument of $\Phi^{(n-2)}$ as
\begin{align*}
\sabs{x-A\sigma}^2 - \sabs{y-A\sigma}^2
&=
\abs{x}^2 - \abs{y}^2 -  2\sinner{A\sigma}{x-y}
\\
&=
\abs{x}^2 - \abs{y}^2 -  2\sinner{\sigma}{Ax-Ay} \,.
\end{align*}
By the Funk-Hecke formula
$\int_{S^{n-1}}  h \kl{ \inner{\sigma}{v}} \rmd \sigma
=
\omega_{n-2}
\int_{-1}^{1} \kl{1-s^2}^{(n-3)/2} h \kl{ \abs{v} s}
\rmd s
$  (see, for example, \cite{Mue66,Nat01}), the inner integral in \req{int1} can be rewritten as
\begin{multline*}
\int_{S^{n-1}}
\Phi^{(n-2)} \kl{\sabs{x - A\sigma}^2
-\sabs{y-A\sigma}^2} \rmd \sigma
\\
\begin{aligned}
&=
\int_{S^{n-1}}
\Phi^{(n-2)} \kl{\abs{x}^2 - \abs{y}^2 -  2\sinner{\sigma}{Ax-Ay} } \rmd \sigma
\\
&=
\omega_{n-2}
\int_{-1}^{1} \kl{1-s^2}^{(n-3)/2}
\Phi^{(n-2)} \kl{ \abs{x}^2 - \abs{y}^2 - 2 \abs{Ax-Ay}s}
\rmd s
\\
&=
\omega_{n-2}
\int_{-1}^{1} \kl{1-s^2}^{(n-3)/2}
\Phi^{(n-2)} \kl{ 2 \abs{Ax-Ay}
\kl{\frac{\abs{x}^2-\abs{y}^2}{2\abs{A\kl{x-y}}}-   s}} \rmd s \,.
\end{aligned}
\end{multline*}
Inserting this  for the inner integral in \req{int1}
yields \req{int}.
\end{proof}

In the following sections we will show that
for special choices of $\Phi$ the kernel in \req{int1}
is a constant multiple of the fundamental solution
of the Laplace equation. For that purpose we will require that
$\skl{\abs{x}^2-\abs{y}^2}/\skl{2\sabs{A\skl{x-y}}} < 1$ for all $x \neq y \in E$.

\begin{lemma}[Simple norm estimate] \label{lem:norm}
For $x , y  \in E$ with $x  \neq y$ we have
\begin{equation*}
	\abs{\frac{\abs{x}^2-\abs{y}^2}{2 \abs{A\kl{x-y}}} } < 1   \,.
\end{equation*}
\end{lemma}

\begin{proof}
By definition we have the inequalities $\sabs{A^{-1}x} < 1$ and $\sabs{A^{-1} y} <1$
for any two points $x$ and $y$ in  the ellipsoid $E$.
From the Cauchy-Schwarz inequality and the triangle inequality we therefore obtain
\begin{equation*}
\abs{\frac{\abs{x}^2-\abs{y}^2}{2\abs{A\kl{x-y}}}}
=
\frac{\abs{\inner{A^{-1} \kl{x+y}}{A \kl{x-y}}}
}{2 \abs{A \kl{x-y}}}
\leq
\frac{\sabs{A^{-1}x} + \sabs{A^{-1}y}}{2} < 1 \,,
\end{equation*}
as we intended to show.
\end{proof}

Recall that  $G_n $ denotes the  fundamental solution  of the Laplacian and  therefore satisfies
$\Delta_x  G_n \kl{x,y} = \delta_n \kl{x-y}$. A simple change of variables implies the following result, that we will also
require in the following sections.

\begin{lemma}[Fundamental solution of $\Delta_{Ax}$] \label{lem:fs}
For every $f  \in  C^\infty \skl{\R^n}$ with compact support and every
$x \in \R^n$, we have the identity
\begin{equation*}
f \kl{x} =  \det\skl{A} \,  \Delta_{Ax}  \int_{\R^n} f \kl{y} G_n \skl{Ax, Ay} \rmd y\,.
\end{equation*}
\end{lemma}

\begin{proof}
Using that $G_n$ is the fundamental solution of
Laplacian and making the  coordinate change  $u = Ay$
shows
\begin{multline*}
\Delta_{Ax}  \int_{\R^n} f \kl{y} G_n \skl{Ax, Ay}
\rmd y
=
\int_{\R^n} f \kl{y} \delta_n \skl{Ax - Ay}
\rmd y
\\
=
\frac{1}{\det\skl{A}}\int_{\R^n} f \kl{A^{-1} u}
 \delta_n \skl{Ax - u}
\rmd u
=
\frac{f \kl{x}}{ \det\skl{A}}
\,.
\end{multline*}
Hence we have verified the required  identity.
\end{proof}

Finally, we will also make use of  the following well
known fact, that the spherical means satisfy
the Darboux  equation:

\begin{lemma}[Darboux equation]\label{lem:darboux}
The spherical means of any $f \in C^\infty \kl{\R^n}$
satisfy the Darboux equation
\begin{equation} \label{eq:darboux}
\Delta_x \Mo f \kl{x, r}
=
\kl{ \frac{\partial^2}{\partial r^2}
+ \frac{n-1}{r} \frac{\partial}{\partial r}  }
\Mo f  \kl{x, r}
\quad \text{ for } \kl{x, r} \in \R^n \times \kl{0, \infty} \,.
\end{equation}
The  right hand side in  \req{darboux} may also be
written as  $\skl{ r^{1-n} \partial_r r^{n-1} \partial_r
\Mo f}  \kl{x, r}$.
\end{lemma}

\begin{proof}
See, for example, \cite{CouHil62}.
\end{proof}

\section{Inversion in even dimension}
\label{sec:even}

Throughout this section, let $n\geq2$ be a given even number.
We will consider the function $\Phi_{\rm even} \in L^1_{\rm loc} \kl{\R}$ defined by
\begin{equation}\label{eq:kernel-even}
 \Phi_{\rm even} \kl{s} \coloneqq  \log \abs{s}
 \quad \text{ for } s \neq 0 \,.
\end{equation}
Further we denote by  $ \Phi_{\rm even}^{(n-2)}$ its $n-2$ fold
distributional derivative.
According to Lemma~\ref{lem:means} we have,  for the
corresponding reconstruction integral defined by \req{bp},
\begin{multline*}
\kl{\Bo_{\Phi_{\rm even}} \Mo f}  \kl{x}
= \frac{\omega_{n-2}}{\omega_{n-1}}\int_{\R^n}
f \kl{y}
\int_{-1}^{1} \kl{1-s^2}^{(n-3)/2}
\\ \times \Phi_{\rm even}^{(n-2)} \kl{ 2 \abs{Ax-Ay}
\kl{\frac{\abs{x}^2-\abs{y}^2}{2\abs{A\kl{x-y}}} -   s}} \rmd s \,
\rmd y
 \,.
\end{multline*}
The reconstruction formula in even dimension essentially  follows from this identity after identifying the inner integral as a constant multiple of the fundamental solution of the Laplacian
in $Ax$.

\begin{lemma}[A Hilbert transform identity]
Denote $\kl{1-s^2}^{(n-3)/2}_+ \coloneqq
\max \set{0, 1-s^2}^{(n-3)/2}$ for $s \in \R$ and recall that
$\Ho_s[\edot] $ denotes the Hilbert transform.
Then, for even $n \geq 4$, we have
\begin{equation} \label{eq:hilbert-even}
\partial_{s_\star}^{n-3}
\Ho_s \ekl{ \kl{1-s^2}^{(n-3)/2}_+  }(s_\star)
= (-1)^{n/2} (n-3)!  \quad \text{ for } \abs{s_\star} < 1 \,.
\end{equation}
\end{lemma}

\begin{proof}
We first note the identity
$\Ho_s \ekl{ s  g }\kl{s_\star} = s_\star  \Ho_s \ekl{g} \kl{s_\star} - \frac{1}{\pi} \int_{\R} g \kl{s} \rmd s$
for $g \colon \R \to \R$; see  \cite[Table~7.3]{Pou10}.  Applying this identity repeatedly
and using the fact that the function $\kl{1-s^2}^{(n-3)/2}$ is equal to $ \sqrt{1-s^2}$ times a  polynomial of degree $n-4$ yields
\begin{equation*}
\Ho_s \ekl{ \kl{1-s^2}^{(n-3)/2}_+ }\skl{s_\star} =
Q_{n-4} \kl{s_\star}
\Ho_s \ekl{ \skl{1 - s^2}^{1/2}_+} \skl{s_\star}
+
Q_{n-5} \kl{s_\star}  \,,
\end{equation*}
for a certain polynomials  $Q_{n-4}$ and $Q_{n-5}$ of degree $n-4$ and $n-5$, respectively. (In the case $n=4$ we take $Q_{n-5} =0$.)
The Hilbert transform of $(1 - s^2)^{1/2}_+$ is  known and equal to $s_\star$ on $\set{\abs{s_\star} < 1}$; see \cite[Table 13.11]{Bra00b}. This implies the identity
\begin{equation*}
\Ho_s \ekl{ \kl{1-s^2}^{(n-3)/2}_+} \kl{s_\star}
=
 c s_\star^{n-3}
+ P_{n-4} (s_\star)   \quad \text{ for } \abs{s_\star} < 1\,,
\end{equation*}
where $c$ is the leading coefficient of $Q_{n-4}$ and $P_{n-4}$ is
a polynomial of degree $\leq n-4$.
The leading coefficient of  $Q_{n-4}$ equals the leading coefficient of $\skl{1-s^2}^{(n-4)/2}$ and is given by $c = (-1)^{(n-4)/2} = (-1)^{n/2}$. Consequently, $\partial_{s_\star}^{n-3}
\Ho_s \bigl[ (1-s^2)^{(n-3)/2}_+ \bigr]$ is constant for $\abs{s_\star} < 1$ and its value is $(-1)^{n/2} (n-3)!$. This
shows \req{hilbert-even}.
\end{proof}

The following Lemma is the key to the reconstruction formula \req{inv-even}.

\begin{lemma}[Kernel in even dimension] \label{lem:even}
For every $x ,y\in E$ with $x \neq y  $ and every $s_\star \in \R$ with
$\abs{s_\star} < 1 $, we have
\begin{multline*}
	\frac{\om_{n-2}}{\omega_{n-1}}
	\int_{-1}^{1} \kl{1-s^2}^{(n-3)/2}
	\Phi_{\rm even}^{(n-2)}
    \kl{2 \abs{Ax-Ay} \skl{s_\star - s}} \rmd s
	\\ =
	(-1)^{(n-2)/2} \,  \frac{ \pi \om_{n-2}  (n-2)!}
	{2^{n-2}} \,  G_n \kl{Ax, Ay} \,.
\end{multline*}
\end{lemma}

\begin{proof}
We first consider the case $n=2$, where the integral to be computed is given by
\begin{multline*}
	\frac{1}{\pi} \int_{-1}^{1} \kl{1-s^2}^{-1/2} \log \abs{ 2 \abs{Ax-Ay} \kl{s_\star -   s}} \rmd s
	 \\ =
	\log \kl{ 2 \sabs{Ax - Ay}}  +
	\frac{1}{\pi} \int_{-1}^{1}
	\frac{\log \abs{s_\star -   s}}{\sqrt{1-s^2} } \rmd s \,.
\end{multline*}
Substituting $s = \cos \alpha$, the latter integral equals
$\frac{1}{\pi} \int_{0}^{\pi} \log \abs{s_\star -   \cos \kl{\alpha}}  \rmd \alpha$. This integral has been computed in \cite{FinHalRak07} and the result is $- \log 2$. Hence the above sum is $\log \kl{ 2 \sabs{Ax - Ay}}  - \log 2
= \log \sabs{Ax - Ay}$. Since $G_2\kl{x,y} =
1/\skl{2\pi} \log \sabs{x-y}$ this shows the claim for the
case $n=2$.
	
Now suppose  $n\geq 4$.  After recalling that
$\Phi_{\rm even} \kl{s} = \log \abs{s}$, that the distributional derivative of $\Phi_{\rm even}$ is
$\mathrm{P.V.} \kl{1/s}$, and that the Hilbert transform is defined as the convolution with $\mathrm{P.V.} \kl{1/\pi s}$ we obtain
\begin{multline*}
\frac{\om_{n-2}}{\omega_{n-1}} \int_{-1}^{1} \kl{1-s^2}^{(n-3)/2}
\Phi_{\rm even}^{(n-2)} \kl{ 2 \abs{Ax-Ay}
\kl{s_\star -   s}} \rmd s
\\
\begin{aligned}
&= \frac{\om_{n-2}}{\om_{n-1} 2^{n-2}  \abs{Ax-Ay}^{n-2}}
\partial_{s_\star}^{n-3}
\int_{-1}^{1} \kl{1-s^2}^{(n-3)/2}
\frac{1}{s_\star - s} \rmd s
\\&=
\frac{\om_{n-2} \pi }{\om_{n-1} 2^{n-2} \abs{Ax-Ay}^{n-2}}
\partial_{s_\star}^{n-3}
\Ho_s \left[ \kl{1-s^2}^{(n-3)/2}_+\right](s_\star)
\\
&= (-1)^{n/2} \frac{\om_{n-2} \pi (n-3)!}
{\om_{n-1} 2^{n-2} \abs{Ax-Ay}^{n-2}}
\\
&= (-1)^{(n-2)/2} \, \frac{\om_{n-2} \pi (n-2)!}
{2^{n-2}} \, G_n \kl{Ax, Ay} \,.
	\end{aligned}
\end{multline*}
For the second last equality  we used Equation \req{hilbert-even} and the last identity follows from the identity  $G_n \kl{Ax,Ay} = - \abs{Ax-Ay}^{2-n}/\skl{\omega_{n-1} (n-2) }$.
The last displayed equation shows the desired identity for the case $n \geq 4$.
\end{proof}

\subsection{Proof of  Theorem \ref{thm:inv-even}}

Lemmas \ref{lem:means}, \ref{lem:norm} and \ref{lem:even}
show
\begin{align*}
\kl{\Bo_{\Phi_{\rm even}} \Mo f}  \kl{x}
& =
\int_{S^{n-1}}
\int_{0}^\infty
r  \D_r^{n-2} r^{n-2} \kl{\Mo f }\kl{A\sigma, r}
\Phi_{\rm even} \kl{\sabs{x - A \sigma}^2 - r^2} \rmd r \rmd \sigma
\\
&=
\frac{\omega_{n-2}}{\omega_{n-1}}
\int_{\R^n}
f \kl{y}
\int_{-1}^{1} \kl{1-s^2}^{(n-3)/2}
\\
&\hspace{6em} \times \Phi_{\rm even}^{(n-2)} \kl{ 2 \abs{Ax-Ay}
\kl{\frac{\abs{x}^2-\abs{y}^2}{2 \abs{A\kl{x-y}}}-   s}} \rmd s
\\
&=
(-1)^{(n-2)/2}
\frac{ \pi \omega_{n-2}  (n-2)!}
{2^{n-2}}
\int_{\R^n}
f \kl{y}
G_n \kl{A x - A y} \rmd y
 \,.
\end{align*}

According to Lemma \ref{lem:fs}, application of
the Laplacian in the variable $Ax$ to the last integral gives $f \skl{x}/\det\skl{A}$. Consequently,
\begin{multline*}
(-1)^{(n-2)/2}
\frac{\pi \omega_{n-2}  (n-2)!}{2^{n-2} \det\skl{A} }
 f (x)
=
\Delta_{Ax} \kl{\Bo_{\Phi_{\rm even}} \Mo f}  \kl{x}
\\
= \Delta_{Ax} \int_{S^{n-1}}
\int_{0}^\infty
r  \D_r^{n-2} r^{n-2} \kl{\Mo f }\kl{A\sigma, r} \log \abs{\sabs{x - A \sigma}^2- r^2} \rmd r \rmd \sigma
\,,
 \end{multline*}
which is the desired reconstruction formula for the even dimensional case, stated in Theorem \ref{thm:inv-even}.

\section{Inversion in odd dimension}
\label{sec:means-odd}

Now let $n\geq 3$ be an odd number.
The proof of the reconstruction formula of Theorem \ref{thm:inv-odd}
will be similar to proof of the corresponding formula in the even dimensional case.
However, the derivation  of the reconstruction formula in odd dimensions is based on the use of the Heaviside function
\begin{equation*}
\Phi_{\rm odd} \colon \R \to \R \colon
s \mapsto \chi \set{s > 0}
\end{equation*}
in place  of the logarithmic function $\log \abs{s}$  used in even dimensions.

The following key Lemma is the counterpart of  Lemma~\ref{lem:even} from the  even dimensional case.  It is however much easier to establish.

\begin{lemma}[Kernel in odd  dimension]\label{lem:odd}
For every $x , y \in E$ with $x \neq y$ and every $s_\star \in \R$ with
$\abs{s_\star} < 1 $, we have
\begin{multline*}
	\frac{\om_{n-2}}{\om_{n-1}}
	\int_{-1}^{1} \kl{1-s^2}^{(n-3)/2}
	\Phi_{\rm odd}^{(n-2)} \kl{ 2 \abs{Ax-Ay}
	\kl{s_\star -   s}} \rmd s
	\\=
	(-1)^{(n-1)/2}
\frac{\om_{n-2}(n-2)!}{2^{n-2}} \, G_n\kl{Ax,Ay}
\,.
\end{multline*}
\end{lemma}

\begin{proof}
The first distributional derivative of $\Phi_{\rm odd}$ is given by the delta distribution, $\Phi_{\rm odd}' \kl{s}= \delta\kl{s}$. Consequently,
\begin{multline*} \label{eq:odd-aux}
	\frac{\om_{n-2}}{\om_{n-1}}
	\int_{-1}^{1} \kl{1-s^2}^{(n-3)/2}
	\Phi_{\rm odd}^{(n-2)} \kl{ 2 \abs{Ax-Ay}
	\kl{s_\star -   s}} \rmd s
\\
\begin{aligned}
&= 		\frac{\om_{n-2}}{\om_{n-1}2^{n-2} \abs{Ax-Ay}^{n-2}}
\partial_{s_\star}^{n-3}
\int_{-1}^{1} \kl{1-s^2}^{(n-3)/2}
\delta\skl{s_\star - s} \rmd s
\\&=
\frac{\om_{n-2}}{\om_{n-1}2^{n-2} \abs{Ax-Ay}^{n-2}}
\partial_{s_\star}^{n-3}
\kl{1-s_\star^2}^{(n-3)/2}
\\&=
(-1)^{(n-3)/2}
\frac{\om_{n-2}(n-3)! }{\om_{n-1}2^{n-2} \abs{Ax-Ay}^{n-2}}
\\&=
(-1)^{(n-1)/2}
\frac{\om_{n-2}(n-2)!}{2^{n-2}} \, G_n\kl{Ax,Ay}
\,.
\end{aligned}
\end{multline*}
For the second last identity we made use of the fact that since  $n$ is odd, $\skl{1-s_\star^2}^{(n-3)/2}$ is a polynomial of degree  $n-3$ for $\sabs{s_\star} < 1$ and for the last
identity we used the representation $G_n \kl{x,y}
= - \abs{x-y}^{2-n}/\skl{\omega_{n-1} (n-2) }$ of the fundamental solution of the Laplacian in $n$ spatial dimensions.
\end{proof}

\subsection{Proof of Theorem \ref{thm:inv-odd}}

Lemmas \ref{lem:means}, \ref{lem:norm} and \ref{lem:odd}
show
\begin{align*}
\kl{\Bo_{\Phi_{\rm odd}} \Mo f}  \kl{x}
&=
\frac{\omega_{n-2}}{\omega_{n-1}}
\int_{\R^n}
f \kl{y}
\int_{-1}^{1} \kl{1-s^2}^{(n-3)/2}
\\
& \hspace{6em} \times \Phi_{\rm odd}^{(n-2)} \kl{ 2 \abs{Ax-Ay}
\kl{\frac{\abs{x}^2-\abs{y}^2}{2 \abs{A\kl{x-y}}}-   s}} \rmd s
\\
&=
(-1)^{(n-1)/2} \frac{\om_{n-2}(n-2)!}{2^{n-2}}
\int_{\R^n} f \kl{y} G_n \kl{A x - A y} \rmd y
 \,.
\end{align*}
Application of Lemma \ref{lem:fs} and recalling that
$\D_r^* r = - r  \D_r $ yield
\begin{multline*}
(-1)^{(n-1)/2} \, \frac{\om_{n-2}(n-2)!}{\det \kl{A} 2^{n-2}}
f (x)
\\
\begin{aligned}
&=
 \Delta_{Ax} \kl{\Bo_{\Phi_{\rm odd}} \Mo f}  \kl{x}
\\
& =
\Delta_{Ax} \int_{S^{n-1}}
\int_{0}^\infty
r  \D_r^{n-2} r^{n-2} \kl{\Mo f }\kl{A\sigma, r} \chi \set{\sabs{x - A \sigma}^2- r^2} \rmd r \rmd \sigma
\\&=
\Delta_{Ax} \int_{S^{n-1}}
\int_{0}^\infty
r  \D_r^{n-3} r^{n-2} \kl{\Mo f }\kl{A\sigma, r}
\delta \kl{\sabs{x - A \sigma}^2- r^2} \rmd r \rmd \sigma
\\&=
\frac{1}{2} \,
\Delta_{Ax} \int_{S^{n-1}}
\int_{0}^\infty
\D_r^{n-3} r^{n-2} \kl{\Mo f }\kl{A\sigma, r}
\delta \kl{\abs{x - A \sigma}- r} \rmd r \rmd \sigma
\\&=
\frac{1}{2} \,
\Delta_{Ax} \int_{S^{n-1}}
\kl{\D_r^{n-3} r^{n-2} \Mo f }\kl{A\sigma,
\abs{x - A \sigma}} \rmd \sigma
\,.
\end{aligned}
\end{multline*}
This shows \req{inv-odd} and concludes the proof of Theorem \ref{thm:inv-odd}.

\section{Discussion}
\label{sec:discussion}

The problem of reconstructing a function from spherical means is important for
many imaging and remote sensing applications
(see, for example, \cite{BelFel09,BurBauGruHalPal07,FinRak09,KucKun08,Nor80,NorLin81,XuWan06}).
Especially in the context of the novel hybrid imaging methods photoacoustic and thermoacoustic tomography
many solution methods have been developed. Known reconstruction techniques can be classified in iterative reconstruction methods
 (see \cite{DeaBueNtzRaz12,DonGoeKun13,PalNusHalBur07b,XuWanAmbKuc04,ZhaAnaPanWan05}),
model based time reversal (see \cite{BurMatHalPal07,FinPatRak04,HriKucNgu08,QiaSteUhlZha11}),  Fourier domain algorithms
(see \cite{AgrKuc07,HalSchZan09b,KoeFraNiePalWebFre01,Kun07b,Kun12,NorLin81,XuXuWan02}), and algorithms based
on explicit reconstruction formulas  of the back-projection 
type (see \cite{And88,Faw85,FinPatRak04,FinHalRak07,Hal11b,Kun07a,Kun11,Nat12,Pal12,XuWan05}).
The  approach implements explicit solutions of the reconstruction problem. It is therefore much faster than iterative solution techniques, where the spherical mean transform and some adjoint transform have to be applied repeatedly.

Explicit back-projection type formulas for recovering a function from spherical mean values on a surface  $S$  are only known for
certain type of surfaces. Such formulas have been derived for a planar surface in \cite{And88,BukKar78,Faw85}  and much later
for spherical surfaces in \cite{FinPatRak04,FinHalRak07,Kun07a,XuWan05}. In two and three spatial dimensions
a formula for certain polygons and polyhedra have been derived in \cite{Kun11} and for ellipses and  ellipsoids in
\cite{AnsFilMadSey13,Hal13a,Nat12}.
Back-projection type formulas for ellipsoids in arbitrary
dimension have been found in \cite{Hal14,Pal12}.
In this paper we derived a new formula for recovering a function from spherical means with centers  on an ellipsoid in arbitrary dimension (see Theorem~\ref{thm:inv-even}) which is different  from the ones in \cite{Hal14,Pal12}. Our reconstruction formula  generalizes one of the formulas of  \cite{FinHalRak07,FinPatRak04}
from the spherical to the elliptical case and can be numerically implemented in an efficient way following, for example,
the implementations presented in \cite{AmbPat07,FinHalRak07} for spherical center sets. Finally, note that we did not touch important theoretical  aspects
such as injectivity results, range conditions, or stability estimates.
Such results have been derived, for example, in  \cite{AgrFinKuc09,AgrKucQui07,AgrQui96,AmbKuc05,FinRak06,
 FriQui14,Pal10}.

\begin{small}
\providecommand{\noopsort}[1]{}

\end{small}

\end{document}